\documentclass[11pt,reqno]{amsart}

\usepackage{diagrams}

\usepackage[utf8]{inputenc}
\usepackage{graphicx}
\usepackage{amssymb}
\usepackage{amsmath}
\usepackage{epstopdf}
\usepackage{amsthm}
\usepackage{mathrsfs}
\usepackage{enumerate}
\usepackage{listings}

\newtheorem{definition}{Definition} 
\numberwithin{definition}{section}
 
\newtheorem{proposition}[definition]{Proposition} 
\newtheorem{theorem}[definition]{Theorem} 
\newtheorem{lemma}[definition]{Lemma} 
\newtheorem{corollary}[definition]{Corollary}

\theoremstyle{definition}
\newtheorem{example}[definition]{Example}

\newcommand{\rar}{\rightarrow}

\newcommand{\RR}{\mathbb{R}}
\newcommand{\NN}{\mathbb{N}}

\newcommand{\ZZ}{\mathbb{Z}}

\newcommand{\SSS}{\mathcal{S}}

\newcommand{\cox}{\operatorname{Cox}}
\newcommand{\ehr}{\operatorname{ehr}}

\renewcommand{\dim}{\mathsf{dim}\ }

\newcommand{\defn}[1]{\emph{#1}}

\renewcommand{\dim}{\operatorname{dim}}

\newcommand{\type}{\operatorname{type}}

\newcommand{\mset}[2]{ \left\{ #1 \; \middle| \; #2 \right\}}

\newcommand{\ncS}{\mathcal{S}}
\newcommand{\ncM}{\mathcal{M}}
\newcommand{\ncL}{\mathcal{L}}

\newcommand{\T}{T} 
\newcommand{\Tn}{T_n} 
\newcommand{\TS}{\T_{S^{n-2}}} 
\newcommand{\TC}{\T_{(0,1)^n}} 

\newcommand{\allow}{\Lambda} 
\newcommand{\poly}{\chi} 
\newcommand{\forb}{\Gamma} 

\begin{document}
\title{Scheduling Problems}
\author{Felix Breuer}
\thanks{Felix Breuer was partially supported by the Deutsche Forschungsgemeinschaft (DFG) grant BR 4251/1-1 and by the Austrian Science Fund (FWF) Special Research Program \emph{Algorithmic and Enumerative Combinatorics} SFB F50-06.}
\address{Research Institute for Symbolic Computation \\ Johannes Kepler University Linz}
\email{felix@fbreuer.de}
\author{Caroline J.\ Klivans}
\address{Departments of Applied Mathematics and Computer Science \\Brown University}
\email{klivans@brown.edu}
\date{\today}

\begin{abstract} We introduce the notion of a scheduling problem which is a boolean function $S$ over atomic formulas of the form $x_i \leq x_j$.  Considering the $x_i$ as jobs to be performed, an integer assignment satisfying $S$ schedules the jobs subject to the constraints of the atomic formulas.  The scheduling counting function counts the number of solutions to $S$. We prove that this counting function is a polynomial in the number of time slots allowed. 
  Scheduling polynomials include the chromatic polynomial of a graph, the zeta polynomial of a lattice, and the Billera-Jia-Reiner polynomial of a matroid.

To any scheduling problem, we associate not only a counting function for solutions, but also a quasisymmetric function and a quasisymmetric function in non-commuting variables. These scheduling functions include the chromatic symmetric functions of Sagan, Gebhard, and Stanley, and a close variant of  Ehrenborg's quasisymmetric function for posets.  

Geometrically, we consider the space of all solutions to a given scheduling problem.
We extend a result of Steingr\'{i}msson by proving that the $h$-vector of the space of solutions is given by a shift of the scheduling polynomial.  Furthermore, under certain conditions on the defining boolean function, we prove partitionability of the space of solutions and positivity of fundamental expansions of the scheduling quasisymmetric functions and of the $h$-vector of the scheduling polynomial.

\end{abstract}
\maketitle


\section{Introduction}

A \defn{scheduling problem} on $n$ items is given by a boolean
formula $S$ over atomic formulas $x_i\leq x_j$ for
$i,j\in[n]:=\{1,\ldots,n\}$. A \defn{$k$-schedule} solving $S$ 
is an integer vector $a \in [k]^n$, thought of as an assignment of the $x_i$,  
 such that $S$ is true when $x_i = a_i$.  We
consider the $n$ items as jobs to be scheduled into discrete time
slots and the atomic formulas are interpreted as the constraints on
jobs.  A $k$-schedule satisfies all of the constraints using at most
$k$ time slots. 

We will be interested in the number of solutions to a given scheduling
problem and define the scheduling counting function $\poly_S(k)$ to be
the number of $k$-schedules solving $S$.  Our first result shows that
$\poly_S(k)$ is in fact a polynomial function in $k$.  As special
instances, scheduling polynomials include the chromatic polynomial of
a graph, the zeta polynomial of a lattice and the order polynomial of a poset.

Our approach to scheduling problems is both algebraic and geometric.
Algebraically, to any scheduling problem, we associate not only a
counting function for solutions, but also a quasisymmetric function
and a quasisymmetric function in non-commuting variables which record
successively more information about the solutions themselves.
Geometrically, we consider the space of all solutions to a given
scheduling problem via Ehrhart theory, hyperplanes arrangements, and
the Coxeter complex of type A.  As special instances, the varying scheduling structures include the
chromatic functions of Sagan and Gebhard~\cite{GS}, and Stanley~\cite{stan-chromatic}, the chromatic
complex of Stein\-gr\'{i}ms\-son~\cite{Ein}, the hypergraph coloring complexes of Breuer, Dall, and Kubitzke \cite{BDK2012}, the P-partition quasisymmetric functions of Gessel~\cite{Gessel}, the matroid invariant of Billera, Jia, Reiner~\cite{BJR},
and a variant of Ehrenborg's quasisymmetric function for posets~\cite{Ehrenborg}.

We first use the interplay of geometry and algebra to prove a Hilbert
series type result showing that the $h$-vector of the solution space
is given by the $h$-vector of a shift of the scheduling polynomial.  This includes and
generalizes Stein\-gr\'{i}mm\-son's result on the chromatic polynomial and
coloring complex to \emph{all} scheduling problems.  Imposing certain
niceness conditions on the space of solutions allows for stronger
results.  We focus on the case when the boolean function $S$ can be
written as a particular kind of decision tree.  Such decision trees
provide a nested if-then-else structure for the scheduling problem.
In this case we prove that the space of solutions is partitionable.
This in turn implies positivity of the scheduling quasisymmetric
functions in the fundamental bases and the $h$-vector of the
scheduling polynomial.

\section{Preliminaries and Scheduling Functions}
\label{sec:scheduling-problems}

\begin{definition}
A \defn{scheduling problem} on $n$ items is given by a boolean
formula $S$ over atomic formulas $x_i\leq x_j$ for
$i,j\in[n]$. A \defn{$k$-schedule} solving $S$ is an integer vector $a \in [k]^n$, thought of as an assignment of the $x_i$,  
 such that $S$ is true when $x_i = a_i$. The \emph{scheduling counting function}
$$ \chi_S(k) := \#k\text{-schedules solving } S$$
counts the number of $k$-schedules solving a given $S$.
\end{definition}

Suppose we are given a scheduling problem with $3$ jobs.
 Any of the three jobs may be started first but different
requirements are imposed depending on which starts first.  If jobs $1$ or $3$ are started first, then the other must start at the same time as job 2. If job $2$ starts first, then job $1$ must occur next before job $3$ can be started.
We interpret the solutions as all those integer points such that
$x_1 < x_2 = x_3$ or $x_3 < x_1 = x_2$ or $x_2 < x_1 < x_3$.
Importantly, solutions only depend on the relative ordering of
coordinates.

There is a natural geometry to the solutions of a scheduling problem which we describe next.  
An \emph{ordered set partition} or \emph{set composition} $\Phi \vDash
[n]$ is a sequence of non-empty sets $(\Phi_1, \Phi_2, \ldots, \Phi_k)$ such
that for all $i,j$, ($ \Phi_i \cap \Phi_j = \emptyset$) and
($\cup_i \Phi_i = [n]$).  The $\Phi_i$ are the blocks of the ordered
set partition and we will often use the notation $\Phi_1| \Phi_2|
\cdots| \Phi_k$.  Note that within each block, elements are not
ordered, so the ordered set partition $13|4|2 \vDash [4]$ is the same
as $31|4|2 \vDash [4]$.  We will use ordered set partitions to represent integer
points whose relative ordering of coordinates is given by the blocks
of the partition.  For example, $31|4|2$ represents all integer points $(x_1, x_2, x_3, x_4)$
such that $x_1 = x_3 < x_4 < x_2$.

The \emph{braid arrangement} $\mathcal{B}_n$ is the hyperplane arrangement in
$\mathbb{R}^n$ consisting of hyperplanes $x_i = x_j$ for all $i,j \in [n]$.
The hyperplanes have a common intersection equal to the line $x_1 = x_2 = \cdots
= x_n$.  Projecting the arrangement to the orthogonal complement of
this line and intersecting with the unit sphere yields a spherical
simplicial complex known as the \emph{Coxeter complex of type $A$}, $\cox_{A_{n-1}}$.
It can be realized combinatorially as the barycentric subdivision of the
boundary of the simplex.

The faces of $\cox_{A_{n-1}}$ are naturally  labeled by ordered
set partitions.  Each non-empty face of the Coxeter complex can be associated to  a face of the cell decomposition induced by
$\mathcal{B}_n$ on $\mathbb{R}^n$.  A face of the cell decomposition of $\mathcal{B}_n$ specifies for each pair $i,j$ whether $x_i < x_j$, $x_i > x_j$, or $x_i = x_j$, precisely the atomic formulas of scheduling problems.  All points in a fixed face have
the same relative ordering of coordinates.  This relative ordering
induces an ordered set partition on $[n]$.  
  Under this correspondence, we see that each maximal face
corresponds to a partition into blocks of size one (i.e., a full
permutation), see Figure~\ref{fig:coxeter}.  Moreover, a face $F$ is contained in a face $G$ if and only if
the ordered set partition corresponding to $F$ coarsens the ordered set partition
corresponding to $G$; the face lattice is dual to the face lattice of the permutahedron.

\begin{figure}[h]
\includegraphics[angle=270,width=4in]{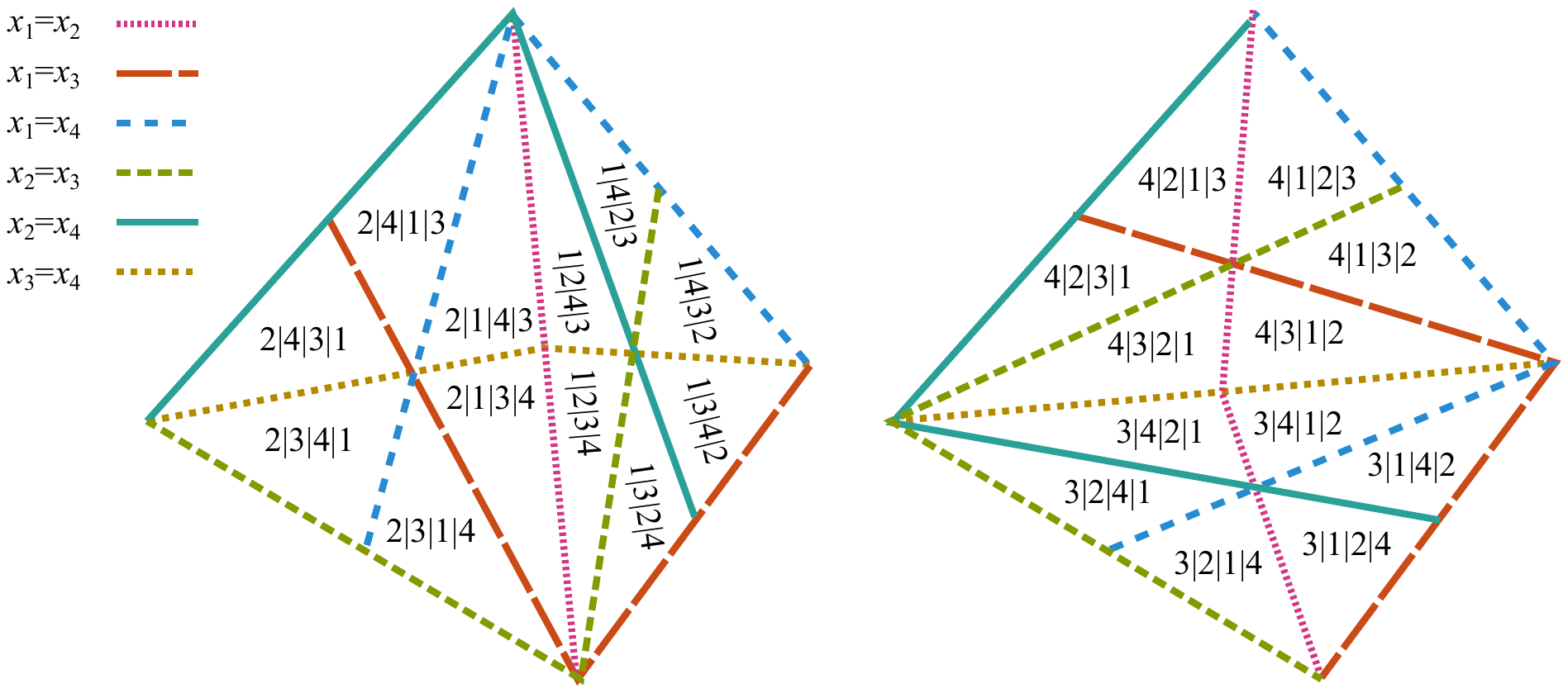}
\caption{ The front and back faces of the Coxeter complex.}
\label{fig:coxeter}
\end{figure}

One of the staples of geometric methods in combinatorics is to interpret a monomial as an integer point in space.  The standard construction is to view a monomial in commuting variables $x_1^{a_1}x_2^{a_2}\dots x_n^{a_n}$ as the point $(a_1, a_2, \ldots, a_n) \in \mathbb{Z}^n$.  In order to develop quasisymmetric functions in non-commuting variables we need a different construction.  
Let $\{x_1, x_2, \ldots \}$ be a collection of non-commuting
variables. For every $a\in\NN^n$, associate the monomial $x_{a_1}\cdots x_{a_n}$ abbreviated as ${\bf x}_a$. For example, $x_2x_1x_3$ corresponds to $(2,1,3)\in\ZZ^3$. The entries of the vector $a$ are given by the \emph{indices} of the monomial, which is well-defined because we are working with non-commuting variables so the factors $x_i$ appear in fixed order. In this way, we can associate to every set of lattice points, $A\subset\NN^n$, a formal sum $N(A)$ of monomials $ N(A) = \sum_{a\in A} {\bf x}_a$. The set of all schedules solving a given scheduling problem $S$ thus corresponds to the generating function
\[
	\ncS_S := \sum_{a\in\NN^n \, : \, a \text{ solves } S} {\bf{x}}_a.
\]

The function $\ncS_S$ has a special structure. Given $a \in \mathbb{N}^n$, let $\Delta(a)$
be the ordered set partition $(\Delta_1 | \Delta_2 | \ldots | \Delta_\ell)$ of $[n]$ such that 
$a$ is the same on each set $\Delta_i$
and satisfies $a|_{\Delta_i} < a|_{\Delta_{i+1}}$ for all $1 \leq i \leq \ell$. 
Define the \emph{order class} of $a$ to be the set of vectors $b$
such that $\Delta(b) = \Delta(a)$.  
For example, for $a = (3,2,2,3,1)$, $\Delta(a) =
5|23|14$ and the order class of $a$ consists of all vectors $x
\in \mathbb{N}^5$ such that $x_5 < x_2 = x_3 < x_1 = x_4$. 
Conversely, an ordered set partition $\Phi$ specifies the relative ordering of coordinates and contains all points in the relative interior of a cone $C(\Phi)=\mset{x\in\RR^n_{>0}}{\Delta(x)=\Phi}$ of the braid arrangement.
The cones $C(\Phi)$ are of the form $C(\Phi)=V\RR^\ell_{>0}$ for matrices $V$ whose columns are called generators and have entries in $\{0,1\}$. The cones $C(\Phi)$ are simplicial; their generators are linearly independent. Moreover, they are unimodular, which means that their fundamental parallelepipeds $V(0,1]^\ell$ contain just a single integer vector.
Crucially, if two vectors $a$ and $b$ have the same order class, then $S(a)\Leftrightarrow S(b)$, that is, either all lattice points in a cone $C(\Phi)$ solve $S$ or none of them do. In the former case, we say that \emph{$\Phi$ solves $S$}. Thus, the solutions to a satisfiable scheduling problem are integer points in a union of these cones, they correspond to a union of open faces of the Coxeter complex.
This geometric phenomenon has an algebraic analogue.

\begin{definition}
A function in non-commuting variables is called quasisymmetric (an element of NCQSym) if
$\forall \, \gamma, \tau \in \mathbb{N}^n$ such that $\gamma$ and $\tau$
are in the same order class ( $\Delta(\gamma) =
\Delta(\tau)$) the coefficient of $x_{\gamma_1}x_{\gamma_2} \cdots
x_{\gamma_n}$ is the same as the coefficient of $x_{\tau_1}x_{\tau_2} \cdots x_{\tau_n}$. We call such functions \emph{nc-quasisymmetric functions} for short.
\end{definition}

The monomial nc-quasisymmetric function $\ncM_\Phi$ indexed by an ordered set partition $\Phi$ is
$$\ncM_{\Phi} := \sum_{a \in \mathbb{N}^n \, \Delta(a) = \Phi} {\bf{x}}_{a}.$$
For example, 
consider the order class of integer points such that the first and third coordinates are equal and less than the second and fourth coordinates which are also equal.  
  The corresponding ordered set partition $\Phi$ is $(13|24)$ and
$$\ncM_{13|24} = x_1x_2x_1x_2 + x_1x_3x_1x_3 + x_2x_3x_2x_3 + x_3x_4x_3x_4 + \cdots$$  
The monomial functions $\ncM_\Phi$ correspond precisely to the sets of lattice points in the cones $C(\Phi)$, see Figure~\ref{fig:cone}, via the function $N$, 
\[
  N(C(\Phi)\cap\ZZ^n) = \ncM_\Phi.
\]
\begin{figure}[t]
\includegraphics[angle=270,width=12cm]{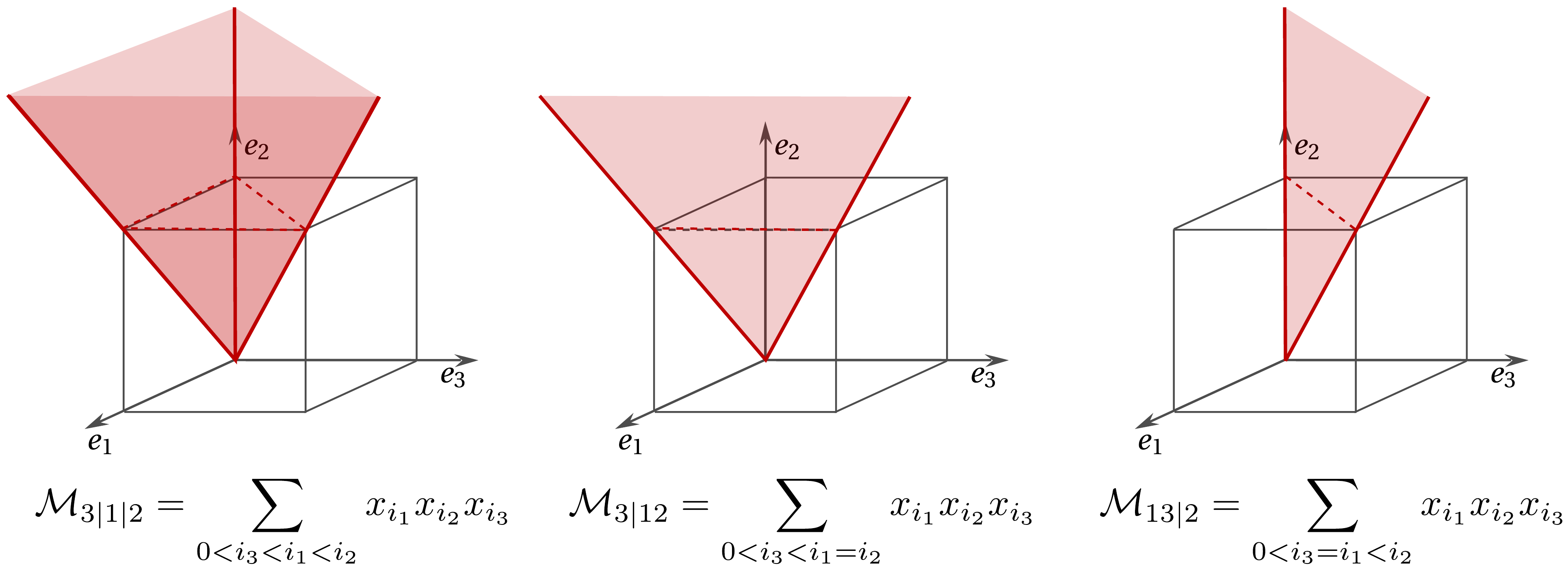}
\caption{\label{fig:cone}Correspondence between $\ncM_{\Phi}$ and $C(\Phi)$.}
\end{figure}
Quasisymmetric functions in non-commuting variables can be expressed as a sum of monomial terms $\ncM_{\Phi}$. We can therefore think of any nc-quasisymmetric function $\mathcal{F}$ with non-negative coefficients in the monomial basis as a multiset of cones, where the multiplicity of lattice points in $C(\Phi)$ is given by the coefficient of $\ncM_\Phi$ in $\mathcal{F}$.
\begin{definition}
\label{scheduling-ncqsym}
Given a scheduling problem $S$ on $n$ items, 
\[
\ncS_S = \sum_{\Phi \text{ solves } S} \ncM_{\Phi}
\]
is an nc-quasisymmetric function, the \emph{scheduling nc-quasisymmetric function} of $S$.
\end{definition}

These observations have a direct impact on the scheduling counting functions $\chi_S$ as well. Informally, an
nc-quasisymmetric function corresponds to a $k$-schedule where $k$ has
been taken to infinity, i.e., there is no deadline.
Imposing a deadline, or restricting to $k$ time slots, corresponds to setting the first $k$ variables $x_1, x_2, \ldots, x_k$ of $\SSS_S$ equal to $1$ and the rest to zero, i.e., $\SSS_S({\bf 1}^k) = \chi_S(k)$. For a single monomial term we have
\begin{equation*}
\label{eqn:ncM-specialization}
\ncM_{\Phi} ({\bf 1}^k) = { k \choose \ell(\Phi) }
\end{equation*}
where $\ell(\Phi)$ is equal to the length, i.e. the number of blocks, of the partition. From Definition~\ref{scheduling-ncqsym} it therefore follows that $\chi_S$ is a linear combination of such binomial coefficients.

From the polyhedral geometry perspective, the substitution of ${\bf 1}^k$ into $\ncM_\Phi$ corresponds to intersecting the cone $C(\Phi)$ with the half-open cube $(0,k]^n$. As Figure~\ref{fig:cone} illustrates, the intersection $C(\Phi)\cap (0,k]^n$ is a half-open simplex. It can be viewed as the $k$-th dilate of a half-open unimodular simplex, since $C(\Phi)\cap (0,k]^n = k\cdot(C(\Phi)\cap (0,1]^n)$, which provides an interesting connection between $\chi_S$ and Ehrhart theory.

For any bounded set $X\subset \RR^n$, the \emph{Ehrhart function} 
$  \ehr_X:\ZZ_{> 0}\rar\ZZ_{\geq 0}$ of $X$ counts the number of integer points in integer dilates of $X$, i.e., $\ehr_X(k) = \#(\ZZ^n \cap k\cdot X$). If $X$ is a polytope whose vertices have integer coordinates, then $\ehr_X(k)$ is a polynomial, called the Ehrhart polynomial of $X$. Two sets $P,Q$ in $\RR^n$ are \emph{lattice equivalent} if there is an affine automorphism $\phi$ of $\RR^n$ with $\phi(P)=Q$ which induces a bijection on the integer lattice $\ZZ^n$. Lattice equivalent sets have the same Ehrhart function. Of special interest to us are the \emph{half-open standard simplices}
\[
    \Delta^n_i = \mset{x\in\RR^{n+1}}{\sum_{i=1}^{n+1} x_i =1, x_1 > 0, \ldots, x_i > 0, x_{i+1} \geq 0, x_{n+1} \geq 0}
\]
of dimension $n$ with $i$ open faces, which have Ehrhart polynomial 
\begin{align}
  \ehr_{\Delta^n_i}(k) = \binom{k+n-i}{n}.
  \label{eqn:half-open-simplex}
\end{align}
If $\Phi$ has $\ell$ parts, then the simplex $C(\Phi)\cap (0,1]^n$ is lattice equivalent to a half-open standard simplex $\Delta^\ell_{\ell}$. The simplex has $k$ open facets and $1$ closed facet, which lies on the closed half of the cube $(0,1]^n$. Its Ehrhart function is thus $\ehr_{C(\Phi)\cap (0,1]^n}(k)=\binom{k}{\ell(\Phi)}$ and $\chi_S$ is the sum over all such Ehrhart functions where $\Phi$ satisfies $S$.
Continuing the example $S(x)=(x_1<x_2 =x_3 \textrm{ \small{or} } x_3 < x_1 = x_2 \textrm{ \small{or} } x_2 < x_1 < x_3)$ gives
 \begin{align*}
 \ncS_S &= \ncM_{1|23} + \ncM_{3|21} + \ncM_{2|1|3},\\
 \chi_S(k) = \ncS_S({\bf 1}^k) &=  2\ehr_{\Delta^2_2}(k) + \ehr_{\Delta^3_3}(k) =  2 {k \choose 2} + {k \choose 3}.
 \end{align*}
 We have now seen both an algebraic and a geometric proof of the  following.
\begin{theorem}
\label{fbasis}
Given a scheduling problem $S$ on $n$ items, the scheduling counting function,
$\poly_S(k)$
is a polynomial in $k$ of degree at most $n$, the \emph{scheduling polynomial} of $S$,
\begin{equation}
\label{eqn:fbasis}
  \poly_S(k) = \sum_{i=1}^n f_i\binom{k}{i},
\end{equation}
 where the coefficients $f_1,\ldots,f_n$ 
are non-negative integers counting the number of ordered set partitions $\Phi$ with $i$ non-empty blocks such that $S(\Phi)$ holds. In particular, the $f_i$ are bounded above by $i!\cdot S(n,i)$, where the $S(n,i)$ are the Stirling numbers of the second kind.
\end{theorem}

Note that the vector $(0,f_1,\ldots,f_n)$ is the $f^*$-vector, as defined in \cite{fstar}, of the Ehrhart function of the subcomplex of the unit cube $(0,1]^n$ that satisfies $S$. We will pursue this perspective further by defining the allowable configuration in the next section. We note that we will have occasion to work with the open cube $(0,1)^n$ and a shift of the Ehrhart polynomial according to $\ehr_{\Delta^\ell_\ell}(k)=\ehr_{\Delta^\ell_{\ell+1}}(k+1)$. These two approaches are interchangeable.

\begin{example}[Graph Coloring]
A particularly familiar example of a scheduling problem is graph
coloring.  Given a finite graph $G = (V,E)$, a $k$-coloring of $G$ is
an assignment $\phi(G): V \rightarrow [k]$ such that for all edges
$\{v_i,v_j\} \in E$, $\phi(v_i) \neq \phi(v_j)$.  Namely, a
$k$-coloring colors the vertices of a graph with at most $k$-colors
such that if two vertices are joined by an edge, then they are given
different colors.  As a scheduling problem, the edges of the graph
give strict atomic formulas: for all edges $\{v_i,v_j\} \in E$, $x_i
\neq x_j$.  

The chromatic nc-quasisymmetric functions are in fact symmetric.  The
chromatic symmetric functions in non-commuting variables were
introduced by Gebhard and Sagan~\cite{GS}.  Allowing the variables to
commute yields the chromatic symmetric function introduced by
Stanley~\cite{stan-chromatic}.  The scheduling counting function is
the well studied chromatic polynomial.  We point out that the
chromatic counting function is usually established to be a polynomial
using a contraction deletion argument on the edges of a graph.  Our
method instead establishes this function as a polynomial via a
specialization of a symmetric function and as an Ehrhart
polynomial. See also
\cite{Beck2006-iop,Breuer2009d,Breuer2011,BDK2012} for further use of
the Ehrhart perspective.
\end{example}

\begin{example}[Order polynomials] Let $(P, \leq)$ be a finite poset.   The order polynomial $\Omega(P;k)$ is the number of order preserving maps from $P$ to $[k]$.  Define a scheduling problem $S$ by taking the conjunction of all relations $x_a \leq x_b$ for $a,b \in P$ with $a\leq b$.  Then  the scheduling polynomial of $S$ is the order polynomial of $P$.  
\end{example}

\section{The space of solutions and Hilbert series}
\label{sec:allow-forb-configurations}
Geometrically, the braid arrangement induces subdivisions
 $\TC$, and $\TS$ of the open unit cube $(0,1)^n$ and of the $(n-2)$-dimensional sphere $S^{n-2}$. 
 The faces of $\TC$ are relatively open unimodular simplices. The
 faces of $\TS$ are sections of the $(n-2)$-sphere. For both, the faces
 are in one-to-one correspondence with the ordered set partitions
 $\Phi$ of $[n]$ into non-empty parts, except in $\TS$ there is no
 face corresponding to to the partition with only one part.
 Combinatorially, $\TC$ is obtained by coning over $\TS$.  
 We will frequently draw no clear distinction between them and refer loosely to the triangulation $\Tn$.
Write $\sigma_{\Phi}\in\Tn$ to denote the face of the triangulation $\Tn$ corresponding to the ordered set partition $\Phi$. 
Define the \emph{allowed configuration} $\allow(S)$ of the triangulation $\Tn$ to be the set of faces 
\[
  \allow(S) = \mset{\sigma_\Phi \in \Tn}{S(\Phi) \textrm{ holds}}.
\]
Correspondingly, define the \emph{forbidden configuration} $\forb(S)$ to be the set of faces
\[
  \forb(S) = \mset{\sigma_\Phi \in \Tn}{\neg S(\Phi) \textrm{ holds}}.
\]

\begin{example}[Coloring Complex] As remarked above, a particularly familiar scheduling polynomial is the chromatic polynomial of a graph $G$.  In this case, the chromatic scheduling problem $S_G$ specifies which variables can not be equal to each other, $x_i \neq x_j$ for $\{v_i,v_j\}$ an edge of the graph, which simply means that no two jobs that are connected by an edge are allowed to run simultaneously.   Steingr\'{i}msson's coloring complex~\cite{Ein} can be described as
the collection of ordered set partitions with at least one edge in at least one block, namely 
the forbidden configuration of the chromatic scheduling problem. In our framework, Steingr\'{i}msson's coloring complex is the forbidden complex of the chromatic scheduling problem taken as a subcomplex of the sphere $S^{n-2}$.
\end{example}
Much work has been done to understand the coloring complex in particular to better understand the chromatic polynomial.  This avenue is possible because of the Hilbert series connection as shown in~\cite{Ein}.  As we will see below this connection holds more generally for \emph{all} scheduling problems. 

Let  $\SSS_{\NN^n}$
denote the nc-quasisymmetric function corresponding to all lattice points in the positive orthant and note that $(k-1)^n$ is the Ehrhart polynomial of the open cube $(0,k+1)^n$. 
If  $S$ is a scheduling problem on $n$ items, then 
 $ \poly_{S}(k) + \poly_{\neg S}(k) = (k-1)^n$ and 
$  \SSS_S + \SSS_{\neg S} = \SSS_{\NN^n}$.

Given numbers $f_0,\ldots,f_n$, the \emph{$h$-vector} $h_0,\ldots,h_{n+1}$ and the \emph{$h^*$-vector} $h^*_0,\ldots,h^*_{n}$ are defined, respectively, via $h_0=1$ and
\begin{align*}
  \sum_{i=0}^{n} f_i \binom{k-1}{i} & = 
  \sum_{i=0}^{n+1} h_i \binom{k+n-i}{n} 
  = \sum_{i=0}^{n} h^*_i \binom{k+n-i}{n}. 
\end{align*}
Typically, the numbers $f_0,\ldots,f_n$ are either the $f$-vector of a (partial) simplicial complex\footnote{That is, $f_i$ counts the number of $i$-dimensional faces of the complex.} or the coefficients of a polynomial $p(k)$ given in the binomial basis $p(k)=\sum_{i=0}^{n} f_i \binom{k-1}{i}$. Given $p(k)$ in this form, the $h$- and $h^*$-vectors can be defined, equivalently, by 
\begin{align*}
1 + \sum_{k=1}^\infty p(k) t^k & =  \frac{\sum_{i=0}^{n+1}h_i t^i}{(1-t)^{n+1}} \;\;\; \text{ and } \;\;\; 
\sum_{k=0}^\infty p(k) t^k  =   \frac{\sum_{i=0}^{n}h^*_i t^i}{(1-t)^{n+1}}. 
\end{align*}

\begin{theorem}
\label{hilbert}
Let $S$ be a scheduling problem on $n$ items. Then the $h$-vector of the shifted scheduling polynomial $\poly_S(k-1)$ is the $h$-vector of the allowed configuration $\allow(S)$ and the $h$-vector of the polynomial $(k-2)^n-\poly_S(k-1)$ is the $h$-vector of the forbidden configuration $\forb(S)$, i.e.,
\begin{align*}
  h(\poly_S(k-1)) &=  h(\allow(S)),   \text{ and}\\
  h((k-2)^n - \poly_S(k-1)) &=  h(\forb(S)),
\end{align*}
or equivalently,
\begin{align*}
 1 + t\sum_{k \geq 0} \poly_S(k) t^k &=  \frac{ h_{(\allow(S))}(t)}{(1-t)^{n+1}}, \text{ and} \\
 1 + t\sum_{k \geq 0} \left( (k-1)^n - \poly_S(k) \right) t^k &=  \frac{ h_{(\forb(S))} (t)}{(1-t)^{n+1}},
\end{align*}
where $h_{\Delta}(t)=\sum_{i} h_i(\Delta) t^i$ is the $h$-polynomial of ${\Delta}$.
\end{theorem}

\begin{proof}
Let $S$ be a scheduling problem, then
\[
  \poly_S(k) = \ehr_{\allow(S)}(k+1) \;\;\; \text{ and } \;\;\; (k-1)^n - \poly_{S}(k) = \ehr_{\forb(S)}(k+1).
\]
Here, the allowed and forbidden configurations $\allow(S)$ and $\forb(S)$ are partial subcomplexes of $\TC$ which is an integral unimodular triangulation of the open unit cube $(0,1)^n$.  In particular, as seen in Theorem~\ref{fbasis}, the coefficients $f_i$ in (\ref{eqn:fbasis}) count the number of $i$-dimensional simplices in the respective partial complexes. Therefore, the $h$-vectors of these partial complexes coincide with the $h$-vectors of their Ehrhart polynomials,
\begin{align*}
 1 + t\sum_{k = 0}^\infty \poly_S(k) t^k &=  1 + \sum_{k = 1}^\infty \ehr_{\allow(S)}(k) t^k = \frac{\sum_{i=0}^{n+1} h_i(\allow(S)) t^i}{(1-t)^{n+1}},\\
 1 + t\sum_{k = 0}^\infty \left( (k-1)^n - \poly_S(k) \right) t^k &=  1 + \sum_{k = 1}^\infty \ehr_{\forb(S)}(k) t^k = \frac{\sum_{i=0}^{n+1} h_i(\forb(S)) t^i}{(1-t)^{n+1}},
\end{align*}
giving the stated results. Analogous statements about the $h^*$-vector are straightforward to derive using the same technique.
\end{proof}

 Theorem~\ref{hilbert}
was established for $\chi_S$ equal to the chromatic polynomial and
$\forb(S)$ equal to the coloring complex in~\cite{Ein}. Our geometric construction specializes to the one given for the chromatic polynomial of a graph in \cite{Beck2006-iop} and for the chromatic polynomial of a hypergraph in \cite{BDK2012}.
Theorem~\ref{hilbert} was also established for all scheduling problems
in which $\forb(S)$ forms a proper subcomplex (i.e. is closed under taking faces) of the Coxeter complex
in~\cite{ABK}.  Note that this is quite restrictive; the solutions to scheduling problems often do not satisfy such closure properties. 

 Theorem~\ref{hilbert} allows one to prove results on
 scheduling polynomials by studying the geometry of the space of
solutions.  We follow this approach in the next section.

\section{Partitionability and Positivity}
\label{geometry}

An important class of results in the literature on quasisymmetric
functions, simplicial complexes and Ehrhart theory are theorems
asserting the non-negativity of coefficient vectors in various bases.
Here we consider the expansions of the scheduling nc-quasisymmetric,
symmetric and polynomial functions in several bases and the relations
of these expansions to the geometry of the allowable and forbidden configuration.  
Specifically, for scheduling problems given by boolean functions of a
particularly nice form, we are able to prove partitionability of the
allowed configuration, which in turn has strong implications for
$h$-vectors.

\subsection{Decision Trees} Decision trees are a very commonly used form of boolean expression. Intuitively, they are simply nested if-then-else statements. We will work with decision trees where the conditions in the if-clauses are inequalities of the form $x_i \leq x_j$ and the conditions in the leaves of the tree are conjunctions of such inequalities (either strict or weak).

\begin{definition} A \emph{leaf} is a boolean expression $\psi$ that is a conjunction of inequalities of the form $x_i \leq x_j$ or $x_i < x_j$. A \emph{decision tree} is a leaf, or a boolean expression of the form
\[
  \text{if $\varphi$ then $\psi_t$ else $\psi_f$}
\]
where $\varphi$ is an inequality of the form $x_i \leq x_j$ or $x_i < x_j$ and $\psi_t,\psi_f$ are decision trees.
\end{definition}

Figure~\ref{fig:decision-tree-with-cells} shows an example of the allowable configuration of a decision tree.
 In general, this region can be non-convex  and even non star-convex.

As decision trees are binary trees, it will be convenient to distinguish notationally between a node $v$ of a tree $S$ and the boolean expression $S_v$ at that node.
Every leaf $v$ of a decision tree corresponds to a conjunction of inequalities which we call the \emph{cell} at $v$. The cell at $v$ is the conjunction of $S_v$ (the inequalities given by the leaf $v$ itself) and the constraints given in the if-clauses of ancestors of $v$, negated according to whether $v$ resides in the ``true" or ``false" branch of the corresponding if-then-else clause. Let $v$ denote a leaf of the tree and let $v_0,\ldots,v_k$ denote its ancestors. Then $S_{v_i}$ is a boolean formula of the form ``if $\varphi$ then $\psi_t$ else $\psi_f$'' and we define $\varphi_{v_i}' := \varphi$ if $v$ resides in the branch $\psi_t$ and $\varphi_{v_i}' := \neg\varphi$ if $v$ resides in the branch $\psi_f$. We denote by $C_v$ the conjunction
\[
 C_v := \varphi_{v_0}'\wedge\ldots\wedge\varphi_{v_k}'\wedge S_v\wedge \bigwedge_{i=1}^n (0 < x_i < 1).
\]
By a slight abuse of notation, we will also use $C_v$ to denote the polytope of all $x\in \RR^n$ satisfying $C_v$. In the above formula, the purpose of $\bigwedge (0 < x_i <1)$ is to ensure that all solutions $x$ lie in the open cube $(0,1)^n$, as is usual in the Ehrhart theory setting. When working with quasisymmetric functions, this condition would be replaced with $\bigwedge (0 < x_i)$ so as to ensure that solutions are positive. In this case the solution sets $C_v$ are cones.

\begin{figure}[h]
\includegraphics[width=1.7in,angle=270]{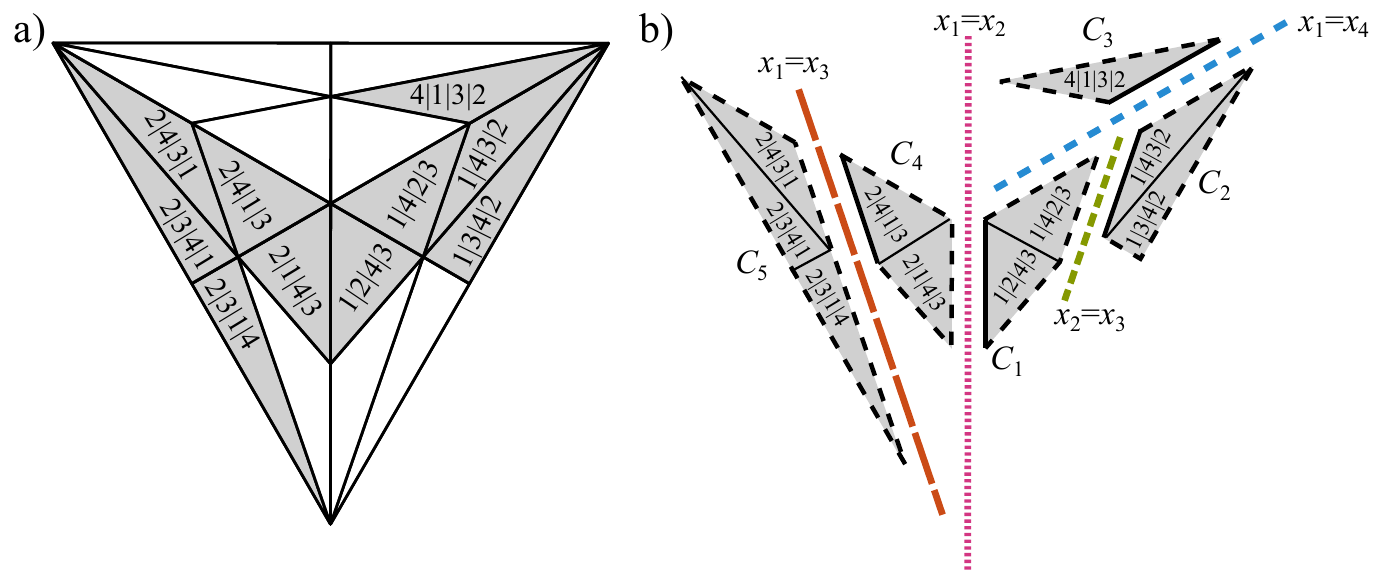}
\caption{Two views of the allowable configuration of a decision tree.}
\label{fig:decision-tree-with-cells}
\end{figure}
To illustrate these definitions, consider the following example which is given in Figure~\ref{fig:decision-tree-with-cells}.
\begin{lstlisting}[mathescape,basicstyle=\sffamily\small]
if $x_1\leq x_2$ then
    if $x_1 < x_4$ then
        if $x_2 < x_3$ then $x_4 < x_3$ else ($x_4 < x_2$ and $x_1 < x_3$)
    else
        $x_3 < x_2$ and $x_1 < x_3$
else
    if $x_1 \leq x_3$ then ($x_2 < x_4$ and $x_4 < x_3$) else ($x_1 < x_4$ and $x_2 < x_3$)
\end{lstlisting}
 The cells of the tree are the following (up to conjunctions of the form $\bigwedge (0 < x_i <1)$).
\begin{align*}
C_1 &= (x_1\leq x_2) \wedge (x_1 < x_4) \wedge (x_2 < x_3) \wedge (x_4 < x_3), \\
C_2 &= (x_1\leq x_2) \wedge (x_1 < x_4) \wedge (x_3 \leq x_2) \wedge (x_4 < x_2) \wedge (x_1 < x_3), \\
C_3 &= (x_1\leq x_2) \wedge (x_4 \leq x_1) \wedge (x_3 < x_2) \wedge (x_1 < x_3), \\
C_4 &= (x_2 < x_1) \wedge (x_1 \leq x_3) \wedge (x_2 < x_4) \wedge (x_4 < x_3), \\
C_5 &= (x_2 < x_1) \wedge (x_3 < x_1) \wedge (x_1 < x_4) \wedge (x_2 < x_3).
\end{align*}
 A disjunction $\bigvee_i C_i(x)$ is called \emph{disjoint}, if the solution sets are disjoint, i.e., if for every $x$ at most one of the $C_i(x)$ is true. The \emph{dimension} of a conjunction $C$ of inequalities is the dimension of the polyhedron of all solutions.
We call a formula $C$ \emph{almost open}, if it is equivalent to a conjunction $\bigwedge_j I_{j}$ of inequalities $I_j$ such that at most one of the $I_j$ is weak and all of the $I_j$ are facet-defining.

We are now ready to show partitionability of the allowed configuration associated to certain decision trees.
We will work in the general setting of partial simplicial complexes.
 A \emph{partial simplicial complex} is a pair
$\Delta=(\bar{\Delta},F)$ where $\bar{\Delta}$ is a simplicial complex
and $F$ an arbitrary subset of the faces of $\bar{\Delta}$. By
convention, we assume that $F$ contains all maximal faces of
$\bar{\Delta}$ and that $F$ does not contain the empty face.
A partial simplicial complex is pure if all maximal faces have the same dimension.

\begin{figure}[h]
\includegraphics[angle=270,width=3.3in]{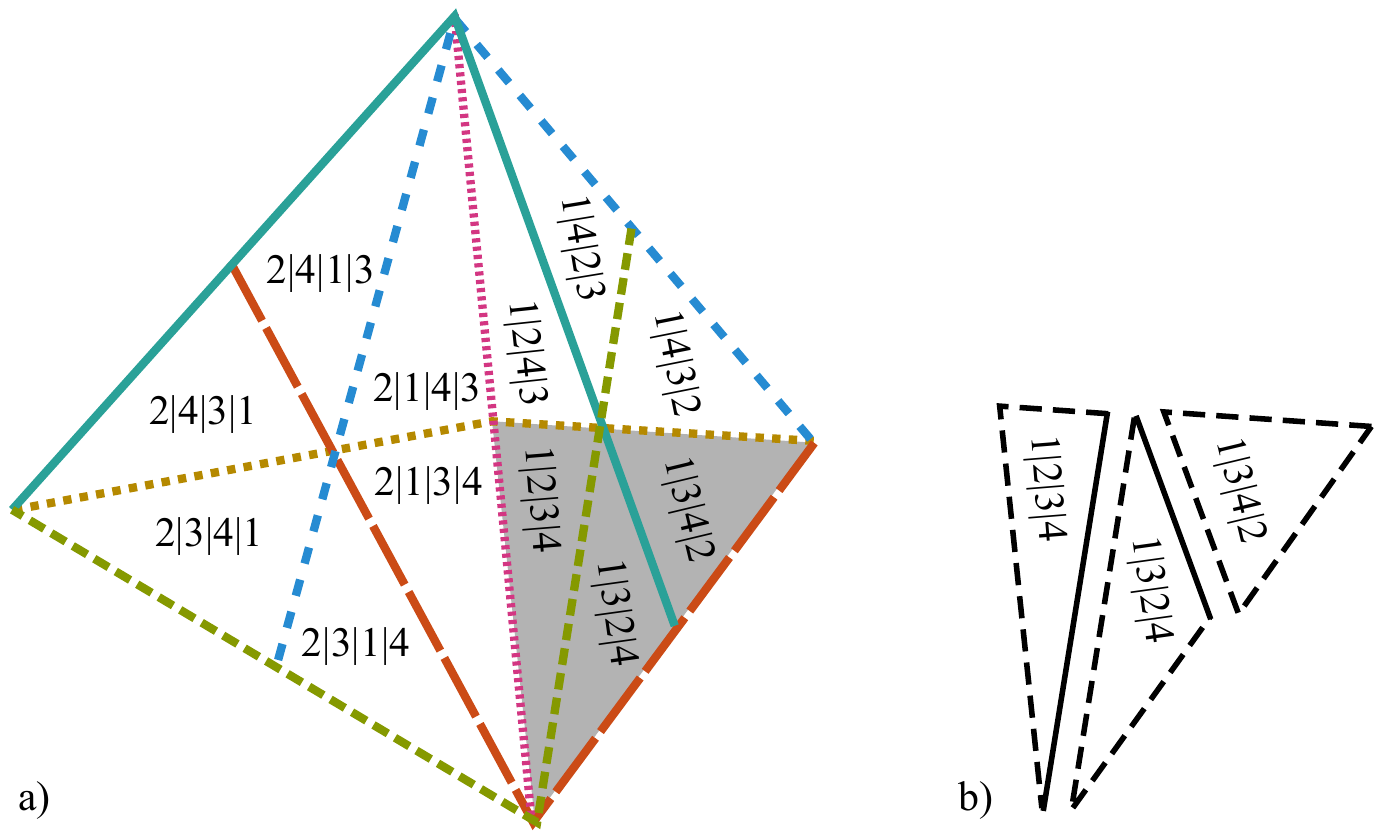}
\caption{
\label{fig:poset}
Partitionability
}
\end{figure}

A pure
 $d$-dimensional partial simplicial complex $\Delta$ is
\emph{partitionable} if it can be written as a disjoint union of half-open
$d$-dimensional simplices. For example, the open region shaded in Figure~\ref{fig:poset}a can be decomposed into half-open simplices as shown in Figure~\ref{fig:poset}b. 
 The $h^*$-vector $(h^*_0,\ldots,h^*_{d+1})$ of a partitionable complex records the
numbers $h^*_j$ of half-open simplices with $j$ open faces in the
partition.
Equivalently, partitionability can also be defined in terms of the face
poset. A pure $d$-dimensional partial simplicial complex $\Delta$ is
partitionable if the face poset can be decomposed as a disjoint union
of intervals $I_i = [l_i, u_i]$ such that for all $i$, $u_i$ has
dimension $d$.  There is a subtlety however working in the setting of
partial complexes; the intervals $I_i$ must partition the set $F$ but
the empty face is ignored.  The $h^*$-vector $(h^*_0,\ldots,h^*_d)$ of
a partition records the numbers $h^*_j$ of intervals $I_i$ with
$\dim(l_i)=j-1$.  To see that these two definitions are equivalent,
note that an interval $I_i$ with $\dim(l_i)=j-1$ is the face poset of
a half-open simplex with $j$ open faces. For closed simplicial
complexes (where $F$ contains all faces), the standard notion of
partitionability \cite{Lee2004,Ziegler1995} can be obtained as a
special case of the above by requiring that $h^*_0=1$.

\begin{theorem}
If a scheduling problem $S$ is equivalent to a formula of the form
\begin{equation}
\label{eqn:disjunction-of-conjunctions}
  S \equiv \bigvee_i C_i \text{ where } C_i = \bigwedge_j I_{i,j}
\end{equation}
where the disjunction is disjoint and the $C_i$ are almost open of dimension $d$, then $\allow(S)$ is partitionable.
\end{theorem}

Because the disjunction is disjoint, this theorem follows immediately, if we can prove it for almost open conjunctions.

\begin{lemma}
\label{lem:almost-open}
An almost open conjunction $C$ of dimension $d$ is partitionable.
\end{lemma}

It is well-known that boundary complexes of simplicial polytopes are partitionable, a fact that can for example be shown using line-shellings \cite{Bruggesser1971}. This method can also be used to construct shellings (and thus partitions) of regular triangulations of polytopes \cite[Corollary~8.14]{Ziegler1995}. These techniques extend naturally to almost open polytopes. For the proof, we assume familiarity with regular triangulations, line-shellings and their connection to partitionability \cite{DeLoera2010,Lee2004,Ziegler1995}.

\begin{proof} 
We first deal with the case where the polytope $C$ has exactly one closed face. This implies that all but one of the facet-defining inequalities of $C$ are strict.

$C$ is a subconfiguration of the braid triangulation $\TC$ of the open cube.
 Therefore, the induced triangulation $T$ of the almost open polytope $C$ is regular. Let $F$ be the facet of $C$ which is closed, and let $z$ denote a new point close to $F$ but outside of $C$.  Define an open polytope $C'$ with $z$ as a new vertex by taking the open convex hull of $C$ and $z$:
\[
  C' = \operatorname{oconv}(C,z) := \mset{\lambda x + (1-\lambda)z}{x\in C, \lambda\in (0,1] }.
\]
 $C'$ is open because $(1-\lambda)<1$ and $F$ is the only closed face of $C$. The extended polytope $C'$ has a triangulation $T'$ which consists of all the simplices in $T$ and the open convex hulls $\operatorname{oconv}(\sigma,z)$ of simplices $\sigma\in T$ such that $\sigma\subset F$ with the new vertex $z$. 

The triangulation $T'$ is regular.\footnote{The new simplices $\sigma\in T'\setminus T$ are separated from $C$ by the hyperplane defining $F$. Moreover $T$ induces a regular triangulation of $F$ and thus the triangulation of $C'\setminus C$ induced by $T'$ is regular as well.} Therefore, there exists a polytope $P$ that has $C'$ as its lower hull. We now construct a line shelling of the lower hull of $P$ which starts with a facet in $T'\setminus T$. The lifted polytope $P$ can be chosen such that that all facets in $T'\setminus T$ are shelled first, again because $T$ is separated from $T'\setminus T$ by a single hyperplane, see also \cite[Lemma~2]{Breuer2011}.

Let $\sigma_1,\ldots,\sigma_N$ be the induced shelling order of the maximal-dimensional simplices in $C'$. Let $\bar{\sigma_i}$ denote the closure of $\sigma_i$. Then, the sequence of half-open simplices
\[
 (\nu_i)_{i=1,\ldots,N}\;  :=\;   \bar{\sigma_1}, \; \bar{\sigma_2}\setminus \bar{\sigma_1}, \; \bar{\sigma_3}\setminus  (\bar{\sigma_2} \cup \bar{\sigma_1}), \; \ldots, \; \bar{\sigma_N} \setminus \bigcup_{i=1}^{N-1} \sigma_i
\]
form a partition of the closed polytope $\bar{C'}$. By reversing which faces of these half-open simplices are open and which are closed, we obtain a partition 
\[
  (\tau_i)_{i=1,\ldots,N}\;  :=\; \bar{\sigma}_N\setminus \partial \bar{C'} ,\;  \bar{\sigma}_{N-1}\setminus (\partial \bar{C'} \cup\bar{\sigma}_N),\;  \ldots,\;  \bar{\sigma}_1 \setminus  (\partial \bar{C'} \cup \bigcup_{i=2}^N \bar{\sigma}_i)
\]
of the open polytope $C'$. Let $\tau_i$ denote the $i$-th element in this sequence.
 A face is open if it occurred in a previous simplex of the sequence or the boundary. Because we run through the sequence in reverse order, this flips the state of all faces versus the first partition.
By construction, the simplices in $T'\setminus T$ are the last half-open simplices in the sequence $(\tau_i)_i$, that is, there exists an $l$ such that $\sigma_i\subset C$ for all $i\leq l$ and $\sigma_i\subset C'\setminus C$ for all $i>l$. Then, the sequence $\tau_1,\ldots,\tau_l$ is a partition of $C$, as desired.

 The proof for the case that $C$ is an open polytope without closed faces is completely analogous, only simpler as there is no need to add the new vertex $z$.
\end{proof}

\begin{corollary}
\label{thm:decision}
Let $S$ be a decision tree such that all cells of $S$ are almost open. Then $\allow(S)$ is partitionable. Moreover, disjoint unions of such decision trees are partitionable.
\end{corollary}

\begin{proof} 
A decision tree is the disjoint union of all its cells. By Lemma~\ref{lem:almost-open}, all cells are partitionable because they are almost open. Therefore the decision tree is partitionable.
\end{proof}

\subsection{Fundamental Bases for NCQSym and QSym}
In this section we prove that partitionability
of the allowable configuration implies positivity of the scheduling
(nc-)quasisymmetric function in the (nc-)fundamental basis, see Theorem~\ref{thm:part-pos}.
First we recall these expansions.

 Let $\Phi_\text{c}$ and $\Phi_\text{f}$ (standing for coarse and fine) be two ordered set partitions such that $\Phi_\text{f}$ is a permutation, i.e., an ordered set partition of maximal length into blocks of size one, that refines $\Phi_\text{c}$. Then the poset of all ordered set partitions $\Phi$  between $\Phi_\text{c}$ and $\Phi_\text{f}$ under the refinement relation forms a boolean lattice of dimension $n-\ell(\Phi_\text{c})$ where $\ell$ is the length of $\Phi_\text{c}$. 
 Thus, as $(\Phi_\text{c};\Phi_\text{f})$ ranges over all such pairs of ordered set partitions, which are comparable under the refinement relation and  where $\Phi_\text{f}$ has length $n$, the functions
\[
  \ncL_{(\Phi_\text{c};\Phi_\text{f})} := \sum_{\Phi_\text{c}\leq \Phi \leq \Phi_\text{f}} \ncM_\Phi
\]
 form a generating system of the linear space of nc-quasisymmetric functions. However, they do not form a basis as there are multiple representations of the same function, for example
\[
  \ncL_{(1|2|3;1|2|3)} + \ncL_{(1|23;1|3|2)} =   \ncL_{(1|23;1|2|3)} +  \ncL_{(1|3|2;1|3|2)} = \ncM_{1|2|3} + \ncM_{1|23} + \ncM_{1|3|2}.
\]
We call this the \emph{fundamental generating system} of the nc-quasisymmetric functions. To obtain a basis, we must fix a choice for $\Phi_\text{f}$ given $\Phi_\text{c}$.  In particular, for any ordered set partition $\Phi$, let $\hat{\Phi}$ denote the permutation refining $\Phi$ with the property that the elements of each part of $\Phi$ are listed in order. For example, if $\Phi=35|247|16$ then $\hat{\Phi}=3|5|2|4|7|1|6$. 
Given an ordered set partition, define
 $\ncL_\Phi:=\ncL_{(\Phi;\hat{\Phi})}.$ As $\Phi$ ranges over all ordered set partitions, the functions $\ncL_\Phi$ form a basis which we call the \emph{nc-fundamental basis}.

Alternatively, this fundamental basis for nc-quasisymmetric functions can be defined in terms of a \emph{directed refinement relation} $\preceq$ on ordered set partitions given by $(\Phi_1, \ldots,\Phi_{i-1}, \Phi_i \cup \Phi_{i+1}, \Phi_{i+2}, \ldots, \Phi_k) \lessdot \Phi$ where every element of the $i$-th block is less than every element of the $(i+1)$-st block.  Then, the nc-fundamental basis for NCQSym can be defined by
$$\mathcal{L}_{\Phi} := \sum_{\Psi:\Phi \preceq \Psi} \ncM_{\Psi},$$
where $\Phi$ ranges over all ordered set partitions. 
We note that this order is opposite to the order $\leq_*$ used to define the basis ${\bf Q_{\Phi}}$ in~\cite{Zab}. 
Our choice of ordering is particularly motivated by its connection to the  fundamental basis $L$ of quasisymmetric functions (QSym); the $\ncL$ basis of NCQSym restricts to the $L$ basis of QSym when the variables are allowed to commute.  
Namely, recall that 
the fundamental quasisymmetric functions of QSym are defined from the monomial quasisymmetric functions as follows.  For any composition $\alpha$, 
$$L_{\alpha} := \sum_{\beta:\beta \textrm{ refines } \alpha} M_{\beta}.$$

 The \emph{type map} maps monomials of NCQSym to QSym by sending ordered set partitions to compositions.
 The type map of an ordered set partition simply records the size of each block:
 $\type(\Phi_1|\Phi_2|\cdots|\Phi_n) = (|\Phi_1|, |\Phi_2|, \cdots, |\Phi_n|).$
 If $\mathcal{S} \in$ NCQSym is written as a sum of monomial terms,
 applying the type map to each index is equivalent to allowing the
 variables to commute.
 Applying the type map in the $\ncL$ basis gives the corresponding quasisymmetric function in the $L$ basis in such a way that directed refinement on the level of nc-quasisymmetric functions corresponds to refinement on the level of quasisymmetric functions:
\begin{diagram}[small]
\text{NCQSym } \ncM & \rTo^{\text{ directed refinement }} & \text{NCQSym } \ncL\\
\dTo^{ \type} & & \dTo_{ \type}\\
\text{QSym } M & \rTo^{\text{refinement}} & \text{QSym } L
\end{diagram}

Let $\Delta$ be an $n$-dimensional half-open unimodular simplex with $i$ open facets, that, as a partial simplicial complex, is a subcomplex of $\TC$.
 The interval of ordered set partitions between $\Phi_\text{f}$ and $\Phi_\text{c}$ corresponds to the face poset of $\Delta$. If $\Phi_{\text{c}}$ is an ordered set partition of length $j$, then 
\[
  \ncL_{(\Phi_\text{c};\Phi_\text{f})}(\mathbf{1}^k) = \ncL_{\Phi_\text{c}}(\mathbf{1}^k) = \binom{k+n-j}{n} = \ehr_{\Delta^n_j}(k).
\]
The geometric reason behind the last inequality, is that monomials in $\ncL_{(\Phi_\text{c};\Phi_\text{f})}(\mathbf{1}^k)$ correspond precisely to the lattice point in the $k$-th dilate of an $n$-dimensional simplex with $j$ open faces in $T_{(0,1]^n}$. The difference to the construction in Section~\ref{sec:scheduling-problems} is that for the fundamental basis, all simplices have the same dimension, but the number of open faces varies.
This observation extends directly to quasisymmetric functions, i.e, we have 
\[
  L_\alpha = \type(\ncL_{\Phi_\alpha}) \;\;\; \text{ and } \;\;\; L_\alpha(\mathbf{1}^k) = \ncL_{\Phi_\alpha}(\mathbf{1}^k) = \binom{k+n-\ell}{n},
\]
 where $\Phi_\alpha$ denotes any ordered set partition with $\operatorname{type}(\Phi_\alpha) = \alpha$, and $\ell$ is the length of $\alpha$.

\begin{proposition}
Let $\mathcal{F}$ denote an nc-quasisymmetric function, let $F$ denote a quasisymmetric function and let $p$ denote a polynomial such that
\[
  \type(\mathcal{F}) = F \;\;\; \text{ and } \;\;\; \mathcal{F}(\mathbf{1}^k) = F(\mathbf{1}^k) = p(k).
\]
  Moreover, let $\mu_\Phi$, $\mu_{(\Phi_\text{c};\Phi_\text{f})}$ and $\lambda_\alpha$ denote the coefficient vectors of $\mathcal{F}$ and $F$ in terms of the fundamental bases and let $(0,h^*_1,\ldots,h^*_n)$ denote the $h^*$-vector of $p$, i.e.,
\begin{align*}
  \mathcal{F} &= \sum_\Phi \mu_\Phi \ncL_\Phi = \sum_{(\Phi_\text{c};\Phi_\text{f})} \mu_{(\Phi_\text{c};\Phi_\text{f})} \ncL_{(\Phi_\text{c};\Phi_\text{f})}, \\
  F &= \sum_{\alpha} \lambda_\alpha L_\alpha,\\
  p(k) &= \sum_{i=1}^n h^*_i \binom{k+n-i}{n}.
\end{align*}
Then
\[
  h^*_i 
  = \sum_{\Phi: \ell(\Phi) = i} \mu_\Phi 
  =
   \sum_{\substack{(\Phi_\text{c};\Phi_\text{f}): 
        \\ \Phi_\text{c} \leq \Phi_\text{f}
        \\ \ell(\Phi_\text{c}) = i
        \\ \ell(\Phi_\text{f}) = n
   }} \mu_{(\Phi_\text{c};\Phi_\text{f})}
  = \sum_{\alpha: \ell(\alpha)=i} \lambda_\alpha.
\]
The coefficients $h_i^*$, $\mu_\Phi$, $\mu_{(\Phi_\text{c};\Phi_\text{f})}$ and $\lambda_\alpha$ are integral but may be negative.  Non-negativity of the $\mu_\Phi$ or the $\mu_{(\Phi_\text{c};\Phi_\text{f})}$ implies non-negativity of the $\lambda_\alpha$, and, in turn, non-negativity of the $\lambda_\alpha$ implies non-negativity of the $h^*_i$.
\end{proposition}

We now bring in the geometry of the allowed and forbidden
configurations. Partitionability implies positivity
expansions for the fundamental bases and the $h^*$ coefficients.

\begin{theorem}
\label{thm:part-pos}
Let $S$ be a scheduling problem such that $\allow(S)$ is partitionable. Then there exists a representation
\[
  \SSS_S = \sum_{(\Phi_\text{c};\Phi_\text{f})} \mu_{(\Phi_\text{c};\Phi_\text{f})} \ncL_{(\Phi_\text{c};\Phi_\text{f})}
\]
with non-negative coefficients $\mu_{(\Phi_\text{c};\Phi_\text{f})}\in \{0,1\}$. In particular, 
the scheduling quasisymmetric function is $L$-positive and the scheduling polynomial is $h^*$-positive.

Conversely, the existence of a representation of $\SSS_S$ with 0-1 coefficients in terms of the fundamental basis implies that $\allow(S)$ is partitionable. 
\end{theorem}

\begin{proof}
As $\allow(S)$ is partitionable, there exists a collection $C$ of half-open $n$-dimensional simplices $\sigma\in\TC$ such that $\allow(S)=\bigcup_{\sigma\in C}\sigma$ and this union is disjoint. Each element $\sigma\in C$ corresponds to a distinct pair $(\Phi_\text{c};\Phi_\text{f})$, where $\Phi_\text{f}$ refines $\Phi_\text{c}$ and $\Phi_\text{f}$ has length $n$. Let $P$ denote the collection of all pairs corresponding to half-open simplices in $C$. Then
\[
  \SSS_S = \sum_{(\Phi_c;\Phi_f)\in P} \ncL_{(\Phi_c;\Phi_f)}
\]
as desired. The non-negativity of the coefficients of the scheduling quasisymmetric function in the fundamental basis and the $h^*$-vector of $\poly_S$ is implied by the existence of a non-negative representation of $\SSS_S$ in the fundamental generating system.  

Suppose there exists a representation of $\SSS_S$ with 0-1 coefficients in terms of the fundamental basis.  For each $\mu_{(\Phi_c;\Phi_f)}$ equal to 1, the pair $(\Phi_c;\Phi_f)$ again corresponds to a half-open simplex. 
Each face of $\allow(S)$ is contained in exactly one pair.  
\end{proof}

\begin{corollary}
Let $S$ be a scheduling problem expressible as a union of decision trees with all cells almost open, then the scheduling quasisymmetric function is $L$-positive and the scheduling polynomial is $h^*$-positive. 
\end{corollary}
\begin{proof}
This is an immediate consequence of Theorem~\ref{thm:part-pos} and Corollary~\ref{thm:decision}.
\end{proof}
The next theorem guarantees positivity of the nc-quasisymmetric scheduling function $\SSS$ in terms of the directed refinement relation.
\begin{theorem}
\label{unique}
Let $S$ be a scheduling problem such that $\allow(S)$ is closed under the directed refinement relation. If for every $\Phi \in \allow(S)$ there exists a \emph{unique} coarsest allowed ordered set partition $\Phi_c \preceq \Phi$ such that $\Phi$ is a directed refinement of $\Phi_c$, then $\allow(S)$ is partitionable and hence the coefficients of $\SSS_S$ in terms of the fundamental basis as well as the $h^*$-vector of $\poly_S$ are non-negative.
\end{theorem}

\begin{proof}
For any  ordered set partition $\Phi_f \in \allow(S)$ of length $n$, there exists by assumption a unique coarsest allowed ordered set partition $\Phi_c$ with $\Phi_c\preceq\Phi_f$. Let $P$ be the collection of all such pairs.  Since $\allow(S)$ is closed under the directed refinement relation, the intervals  $[\Phi_c, \Phi_f]$ are boolean lattices.
Furthermore,  for any  $\Phi \in \allow(S)$ there exists a $\Phi_f$ of length $n$ that is a directed refinement of $\Phi$, hence $\Phi$ is contained in the interval with $\Phi_f$ as its maximal element. Therefore, $P$ forms a partition of $\allow(S)$  which completes the proof.
\end{proof}

Note that the condition that $\allow(S)$ is closed under the directed
refinement relation is equivalent to the requirement that the forbidden
configuration $\forb(S)$ be a valid subcomplex of the Coxeter complex; i.e., a
collection of faces closed under taking subsets.
An important class of scheduling problems that satisfy the conditions of Theorem~\ref{unique} are those scheduling problems that can be expressed as a disjunction of conjunctions of strict inequalities, i.e., scheduling problems of the form (\ref{eqn:disjunction-of-conjunctions}) where the $I_{i,j}$ are strict inequalities.  Such
scheduling problems are a special case of
decision trees with all cells fully open, whence all regions of the allowable configuration are convex. Therefore, Theorem~\ref{thm:part-pos} already
provides $\ncL$-positivity, but the conditions of
Theorem~\ref{unique} are easier to interpret in this case.  The
$C_i$ are conjunctions of strict inequalities. Geometrically, the
allowable schedules given by a $C_i$ form a cone; the intersection of
halfspaces defined by the inequalities $x_a < x_b$.
 On the Coxeter
complex, such regions are known as posets of the complex.  In particular, for a given $C_i =
\bigwedge_j I_{i,j}$, the inequalities $I_{i,j}$ naturally induce a
partial order on $[n]$.  The collection of facets of the Coxeter
complex (thought of as permutations) contained in $C_i$ consist of all possible linear extensions of the partial
order.

\begin{example}[P-partitions]
Let $P$ be a poset and $\omega:P\rar\NN$ a labeling of the elements of $P$.  Define a scheduling problem by constructing a conjunction $C_{\omega}$ as follows. For every covering relation $a <_P b$, $C_{\omega}$ contains the weak inequality $x_a \leq x_b$ if $\omega(a) < \omega(b)$ and the strict inequality $x_a < x_b$ if $\omega(b) < \omega(a)$.  
The resulting allowable configuration consists of all $(P,\omega)$-partitions and is the half-open order polytope defined by $\omega$, see~\cite{Stanley1974}.
The scheduling quasisymmetric function is the P-partition generating function
$K_{(P,\omega)}$ defined by Gessel~\cite{Gessel}.  
The ``fundamental theorem of quasisymmetric functions'' ~\cite{Gessel, StanP} is the expansion of 
 $K_{(P,\omega)}$ in the fundamental basis in terms of descent sets of linear extensions of $P$.  
The descent sets provide the unique coarsest
elements of Theorem~\ref{unique}.
\end{example}

\begin{example}[The Coloring Complex]
The coloring complex, regarded as the forbidden
subcomplex of the graph coloring problem, is a valid subcomplex of the Coxeter complex.  
 The graphical arrangement $\mathcal{A}_G$ associated to $G$ is the
subarrangement of the braid arrangement consisting of the hyperplanes
$\{x_i = x_j | \, v_i, v_j \in E\}$. The graphical zonotope ${P}_G$ is the
zonotope dual to this arrangement formed by the sum of all normals to
all planes in the arrangement.  Geometrically, this leads to a
perspective first noted explicitly by Hersch and Swartz: the coloring
complex of $G$ is the codimension one skeleton of the normal fan of ${P}_G$
as subdivided by the Coxeter complex~\cite{HS}.  Equivalently, as a scheduling problem,  the
allowable configuration consists of all integer points in the
interiors of maximal cones of the normal fan.  These interiors are
conjunctions of strict inequalities, one for each facet defining
hyperplane of the cone.
\end{example}

\begin{example}[Generalized Permutahedron]
  A scheduling problem
$S$ can be associated to any generalized permutahedron $GP$ by
defining the forbidden configuration $\forb(S)$ to be the codimension
one skeleton of the normal fan as subdivided by the Coxeter
complex.  Such a scheduling problem is then given as a disjunction of
conjunctions.
 The valid  schedules correspond to all
integer points in the interior of the normal fan.  In~\cite{ABK}, this allowable configuration and nc-quasisymmetric function were studied not as a scheduling problem but in connection to the Hopf monoid of generalized permutahedron. It
was shown that the  generalized permutahedron nc-quasisymmetric function is $\ncL$-positive and $\forb(S)$ is $h$-positive.
\end{example}

\begin{example}[Matroid Polytopes]
Returning to the graphical case, again in~\cite{HS}, the perspective of the normal fan is used to prove that the coloring
complex has a convex ear decomposition which implies strong relations on
the chromatic polynomial. The authors consider the generalization of their results to characteristic polynomials of matroids.
 They note empirically
however that the result do not seem to generalize.  The perspective
here suggests that the generalization should not be from chromatic
polynomials to characteristic polynomials, but from the scheduling
polynomials of graphic zonotopes to the scheduling polynomials of
matroid polytopes. 
 The corresponding scheduling polynomial for matroid polytopes is the
polynomial restriction of the Billera-Jia-Reiner quasisymmetric
function for matroids~\cite{BJR}.
\end{example}

The special cases above are scheduling problems in which the forbidden configuration,
$\forb(S)$, is  a valid subcomplex of the Coxeter complex.
 Next we consider scheduling problems such that the allowed configuration, $\allow(S)$, is a valid
 subcomplex of $\cox_{A_{n-1}}$, i.e., the ordered set partitions satisfying
 $S$ are closed under coarsening.     
 In this case, expanding the scheduling quasisymmetric functions in the fundamental bases is not a natural
 choice.  Expansion in the \emph{co-fundamental} bases, however, is
 natural and does yield good behavior.  The co-fundamental basis for
 NCQSym is defined analogously to the $\mathcal{L}$ basis above using
 a directed coarsening relation.  This basis was first defined
 in~\cite{Zab} and denoted ${\bf {Q}_{\Phi}}$.  Allowing the variables
 to commute gives the co-fundamental basis for QSym.

 Our examples in which $\allow(S)$ is a subcomplex correspond to collections of flags.  Given an integer point $a$ or an ordered set partition $\Delta(a)$ we associate a flag:
$$\mathscr{F}_{a} = \mathscr{F}_{\Delta(a)} := F_0 \subset F_1 \subset \ldots \subset F_k := [n]$$
such that $F_i - F_{i-1} = \Delta(a)_i$ and  $a|_{F_i -
  F_{i-1}} < a|_{F_{i+1} - F_{i}} $. For instance, 
suppose $a$ satisfies $a_2 < a_1 = a_3 = a_4$ then $\Delta(a) = (2|134)$ and $\mathscr{F}_{\Delta(a)} = {\emptyset \subset \{2\} \subset \{1,2,3,4\}}$.

 \begin{example}[Graded Posets and Ehrenborg's quasisymmetric function]
Let $S$ be a scheduling problem such that the collection of flags corresponding to the elements of $\allow(S)$,
$\{ \mathscr{F}_{\Delta(a)} \, | \, \Delta(a) \in \allow(S) \},$
forms the collection of all flags of a graded poset.  Then $\allow(S)$ is closed under coarsening and
 the coefficient of $N_{\alpha}$  of the scheduling quasisymmetric function in the co-fundamental basis is given by:
\begin{equation}\label{posets}[N_{\alpha}] = (-1)^{|\alpha|} \sum_{\alpha \, \textrm{flags}} \prod_i \mu(f_i,f_{i+1}),\end{equation}
where an $\alpha$-flag is a flag such that $|f_{i+1}| - |f_i| = \alpha_i$ and $\mu$ is the M\"{o}bius function.

This scheduling quasisymmetric function is a variant of Ehrenborg's quasisymmetric function for graded posets~\cite{Ehrenborg}.  
The scheduling quasisymmetric function is indexed
by compositions recording the \emph{size} of each step in the flag.  
For any graded poset $P$, Ehrenborg defines
a quasisymmetric function $F(P)$ by summing over all chains of the
poset and recording the \emph{rank jump} of the flag at each step.  Although
the quasisymmetric functions record different data from the poset, 
Equation~\ref{posets} is equivalent to~\cite[Proposition~5.1]{Ehrenborg}.  Our expansion in the
co-fundamental basis is a rephrasing of Ehrenborg's expansion of the image of
the Malvenuto and Reutenauer involution of quasisymmetric functions.
We do not reproduce his proof here.  We simply note that Ehrenborg's derivation of 
 the coefficients is given by a
manipulation of the M\"obius function, and the manipulation continues to hold 
for any collection of compositions associated to 
chains that is closed under coarsening.
\end{example}

\begin{example}[Lattices]
Further suppose that $S$ is a scheduling problem such that the
collection of flags corresponding to the elements of $\allow(S)$ form
the collection of all flags of a lattice $L$.  Then the scheduling
polynomial $\poly_S(k)$ is the \emph{zeta polynomial} of the lattice
which counts the number of multichains of length $k$,
$$\poly_S(k) = |\{ \hat{0} = y_0 \leq y_1 \leq \cdots \leq y_k =
\hat{1} \, | \, y_i \in L \}|.$$
\end{example}

\begin{example}[The lattice of flats and the Bergman fan]
  Let $M$ be a matroid and $L(M)$ be the lattice of flats of $M$.
  Consider the scheduling problem $S$ such that the flags
  corresponding to $\allow(S)$ are precisely the flags of flats of
  $M$.  In~\cite{AK}, it was shown that $\mathscr{F}_{\Delta(a)}$
  is a flag of flats of $M$ if and only if the integer points of order
  class $\Delta(a)$ are in the \emph{Bergman fan} of $M$ ~\cite[Section 9.3]{Sturmfels}.  Thus the
  Bergman fan can be seen as an allowable configuration of a
  scheduling problem.  Briefly, scheduling solutions induce weight
  functions such that all elements of the matroid are contained in
  minimum weight bases.

  As above, the scheduling polynomial is the zeta-polynomial of the
  lattice of flats and counts multichains of flats of length $k$.  One
  can interpret the matroid rank function as a kind of cost function -
  once certain jobs are started, others of the same rank can be added
  without additional cost. To minimize cost, we require that in any
  scheduling of jobs, at each time step we have a closed subset of
  jobs.

  \end{example}

\section*{Acknowledgments}
The authors thank Louis Billera, Matthias Beck and Peter McNamara for many helpful discussions.

 \bibliographystyle{amsplain}

 \bibliography{references}

\end{document}